\documentclass[a4paper,12pt]{article}

\usepackage{mathtools}
\usepackage{amsmath,amssymb,amsthm, bm}
\usepackage{enumerate}
\usepackage[noadjust]{cite}

\usepackage[spanish, english]{babel}

\usepackage[ansinew]{inputenc}
\usepackage{here}
\usepackage{hyperref}

\usepackage{xcolor}

\usepackage{graphicx}

\usepackage[normalem]{ulem}

\newcommand{\tp}{\operatorname{tp}}

\newcommand{\FF}{\mathcal{F}}
\newcommand{\GG}{\mathcal{G}}

\newcommand{\MM}{\mathcal{M}}
\newcommand{\NN}{\mathcal{N}}

\newcommand{\Mm}{\mathcal{M}}
\newcommand{\Nn}{\mathcal{N}}

\newcommand{\Right}{Q}
\newcommand{\Left}{P}

\newcommand{\ar}{\rangle}
\newcommand{\al}{\langle}

\newcommand{\lleq}{\preceq}

\newcommand{\piBN}{\pi(B)^{\Nn}}
\newcommand{\BNx}{B_t^{\Nn}}

\theoremstyle{definition}
\newtheorem{definition}{Definition}[section]

\newtheorem{example}[definition]{Example}

\theoremstyle{plain}
\newtheorem{theorem}[definition]{Theorem}

\newtheorem{lemma}[definition]{Lemma}

\newtheorem{claim}{Claim}[definition]

\newtheorem*{question*}{Question}
\newtheorem{question}[definition]{Question}

\newenvironment{claimproof}[1][\proofname]
               {
                 \proof[#1]
                 
               }
               {
                 \endproof
               }

\title{The Marker-Steinhorn Theorem}
\author{Pablo And\'ujar Guerrero\footnote{University of Leeds} \\
\small \emph{Email address:} pa377@cantab.net}
\date{}

\begin{document}
\maketitle

\noindent
{\small \emph{2020 Mathematics Subject Classification:} 03C64. } 

\begin{abstract}

    We give a proof of the Marker-Steinhorn Theorem which fills a gap in previous proofs of the result.
\end{abstract}

\section{Introduction}

\subsection{The Marker-Steinhorn Theorem}

Let $\MM=(M,<,\ldots)$ be an o-minimal structure. An elementary extension $\NN=(N,<,\ldots)$ of $\MM$ is \emph{tame}\footnote{In~\cite{mark_stein_94} they write that $\MM$ is \emph{Dedekind complete} in $\NN$} if, for every $a\in N$, the set $\{ b\in M : b < a\}$ has a supremum in $M\cup \{-\infty, +\infty\}$. 
The Marker-Steinhorn Theorem~\cite[Theorem 2.1]{mark_stein_94} states that a type $p \in S_n(M)$ is definable if and only if it is realized in a tame elementary extension $\NN$ of $\MM$. In particular, for any $N$-definable subset $X$ of $N^{n}$, the externally definable set $X \cap M^n$ is $M$-definable. 


The Marker-Steinhorn Theorem was first proved in~\cite{mark_stein_94} by the namesake authors. This was followed shortly after by a proof by Pillay~\cite{pillay94}.
Tressl~\cite{tressl} gave a short non-constructive proof for o-minimal expansions of ordered fields via valuation theory. Van den Dries~\cite{dries03} produced a short proof of a stronger version of the theorem, again in the case of o-minimal expansions of ordered fields. Chernikov and Simon~\cite{cher-sim-2} state the theorem for o-minimal expansions of Dedekind complete linear orders as a corollary of a more general theorem around stable embeddedness in the NIP setting. By analizing definable linear orders, Walsberg~\cite{walsberg19} gave a proof for o-minimal expansions of ordered groups.


In the present paper we present a proof of the Marker-Steinhorn Theorem which fills a gap in the proofs of the full result in~\cite{mark_stein_94} and~\cite{pillay94}. Our approach also streamlines these proofs, circumventing the structure by cases in~\cite{mark_stein_94}, as well as the use of regular cell decomposition in both proofs.

In Section~\ref{sec:conv} we fix terminology and conventions. In Section~\ref{sec:gap} we describe the gap in the proofs presented in~\cite{mark_stein_94} and~\cite{pillay94}. In Section~\ref{remark:preorder} we prove two results on definable preorders in o-minimal structures which we will require in our proof. Finally, in Section~\ref{subsec:proofMS} we prove the Marker-Steinhorn Theorem. We end the paper by asking (Question~\ref{question:dries}) whether the strengthening of the theorem for o-minimal expansions of ordered fields proved by van den Dries~\cite{dries03} holds in every o-minimal theory.


\subsection{Conventions} \label{sec:conv}


Throughout we work in an o-minimal structure $\MM=(M,<,\ldots)$ and an elementary extension $\NN=(N,<,\ldots)$.
Throughout $n$, $m$, $k$ and $l$ are natural numbers greater than zero.

Any formula is in the language of $\MM$ without parameters unless stated otherwise.
We use $u$ and $v$ to denote tuples of variables. We use $a, b, c, d, e, x$ and $y$ to denote tuples of parameters. We use $s$, $t$ and $r$ exclusively for unary variables or parameters. We denote by $|u|$ the length of a tuple $u$.

For any (partitioned) formula $\varphi(u,v)$, let $\varphi^{\text{opp}}(v,u)$ denote the same formula after switching the order of the variables $v$ and $u$. For any formula $\varphi(u)$, possibly with parameters from $N$, and set $A\subseteq N^{|u|}$, let $\varphi(A)=\{a \in A : \NN\models \varphi(a)\}$. Throughout and unless otherwise specified ``definable" means ``definable in $\Mm$ over $M$".

We use notation $\al a,b\ar$ for ordered pairs, setting aside the notation $(a,b)$ for intervals. 
For a set $B \subseteq N^{n}$ and an element $x \in N^{m}$, with $0<m<n$, we denote the fiber of $B$ at $x$ by $B_x=\{ t\in N^{n-m} : \al x,t \ar \in B\}$.


In general an $n$-type is a type in $S_n(M)$, interpreted as a non-trivial ultrafilter in the Boolean algebra of definable subsets of $M^n$. Recall that an $n$-type $p$ is definable if, for every formula $\varphi(u,v)$ with $|u|=n$, the set $\{ b \in M^{|v|} : \varphi(M^n, b) \in p\}$ is definable. By \emph{type basis} for a type $p$ we mean a filter basis, meaning a subset $q \subseteq p$ such that any set in $p$ is a superset of some set in $q$.

We fix some standard notation regarding o-minimal cells. For a function $f$ we denote its domain by $dom(f)$. Let us say that a partial function $M^{n}\rightharpoonup M\cup \{-\infty, +\infty\}$ is definable if either it maps into $M$ and is definable in the usual sense or otherwise it is constant and its domain is definable. Given two such functions $f$ and $g$, let $(f,g)=\{\al x, t \ar : x \in dom(f) \cap dom(g) \text{ and } f(x) < t <g(x)\}$. For a set $B\subseteq M^{n}$, let us say that a family $\{f_b : b\in B\}$ of definable partial functions $M^{n}\rightharpoonup M\cup\{-\infty, +\infty\}$ is definable if the sets $B(+\infty)=\{ b\in B : f_b \equiv +\infty\}$ and $B(-\infty)=\{ b\in B : f_b \equiv -\infty\}$ are both definable and moreover the families $\{ dom(f_b) : b\in B(+\infty)\}$, $\{ dom(f_b) : b\in B(-\infty)\}$ and $\{f_b : b\in B \setminus (B(+\infty)\cup B(-\infty))\}$ are (uniformly) definable in the usual sense. We adopt the same conventions for functions and families of functions $N^{n}\rightharpoonup N\cup \{-\infty, +\infty\}$. 

We direct the reader to~\cite{dries98} for background o-minimality. In particular we will make use of o-minimal cell decomposition~\cite[Chapter 3, Theorem 2.11]{dries98}.

\subsection{The gap}  \label{sec:gap}

Let $\Nn$ be a tame elementary extension of $\Mm$. Consider an $n$-type $p(v) \in S_n(M)$, with $n>1$, that is realized in $\Nn$ by an element $a\in N^{n}$. We write $a=(d,e)\in N^{n-1}\times N$. 

Let $\varphi(v,u)$ be a formula with $|u|=2$ satisfying that, for every $b\in N^2$, the set $\varphi(N^n,b)$ is an o-minimal cell of the form $(f_b, +\infty)$ for some partial function $f_b:N^{n-1}\rightharpoonup N$. 

Let $B=\psi(M^2)$ be a subset of $M^2$, defined by a formula $\psi(u)$, with the property that, for every $b\in B$, it holds that $\Nn\models \exists t \varphi(d,t,b)$. 
In proving the Marker-Steinhorn Theorem (Theorem~\ref{thm:marker_steinhorn}), we wish to show that $\varphi(a,B)=\{ b\in B : f_b(d)<e\}$ is definable. We label this set $P = \varphi(a,B)$. 

Let $B^{\Nn}=\psi(N^2)$.
Let $\pi:N^2\rightarrow N$ be the projection to the first coordinate.  Recall our fiber notation. We may make the following assumptions. 
\begin{enumerate}[(i)]
    \item $B$ is an open cell.  
    \item For every $t \in \pi(B)$, either $B_t = P_t$ or $B_t \cap P_t=\emptyset$. 
    \item For every $t \in \pi(B)$, there exists $r_1, r_2\in B^{\Nn}_t$ such that $f_{\al t,r_1  \ar}(d)<e$ and $f_{\al t,r_1 \ar}(d)>e$. Hence $\Nn\models \varphi(a,t,r_1)$ and $\Nn \not\models \varphi(a,t,r_2)$.
\end{enumerate}
(One may additionally assume that, for every $t\in \pi(B)$ and $r, r'\in B^{\Nn}_t$ with $r < r'$, if $d$ is in the domain of $f_{\al t,r \ar}$ and $f_{\al t,r' \ar}$, then $f_{\al t,r\ar}(d)<f_{\al t,r'\ar}(d)$.)

\begin{example}
An example with $B=M^2$ in the case where $\Mm$ expands an ordered field would be as follows. For $v=(v_1,v_2)$ and $u=(u_1,u_2)$ binary variables, let $\varphi(v,u):= v_2>u_1 v_1 + u_2$. Let $p(v)$ a $2$-type realized by an element $\al d, e\ar$ such that $d>r$ for every $r\in M$ and, furthermore, for any $\al b_1, b_2 \ar \in M^2$ the inequality $e>b_1 d + b_2$ holds if and only if $b_1 <1$. Note that in this case we have that $P=(-\infty, 1)$.  
\end{example}

Proving the definability of $P$ is addressed indirectly in Claim (2.8.2)(ii) (page 193) in~\cite{mark_stein_94}, and at the end of the proof of Theorem 1.1. in~\cite{pillay94}. However in both cases the proofs have gaps.
In our proof of Theorem~\ref{thm:marker_steinhorn} this situation is dealt with towards the end, through the treatment of the case $I^*=I^*(2)$. 
Below we describe two specific scenarios where the proof that $P$ is definable is relatively straightforward, even when we generalize the scenario to having $P\subseteq M^m$ for any $m>1$, and $\pi:M^m \rightarrow M^{m-1}$ being the projection to the first $m-1$ coordinates.

If there exists a definable function $h:\pi(B)\rightarrow M$ with $h(t)\in B_t$ for all $t\in \pi(B)$, then we may apply an induction (on dimension) argument to the image of $h$ and derive that its intersection with $P$ is definable, and hence, using assumption (ii), that $P$ is definable. This applies to any case where $\Mm$ has definable Skolem functions (e.g. $\Mm$ expands an ordered group), as well as the case where for example $B=M^2$ (by taking $h$ to be the identity map). 

If $\Mm$ expands $(\mathbb{R},<)$ then, since by Dedekind completeness every elementary extension is tame, the Marker-Steinhorn Theorem simply states that every externally definable set is definable. In this case we may use the fact, proved by Shelah~\cite{shelah09}, that the projection of every externally definable set in an NIP (e.g. o-minimal) structure is externally definable, to observe that $\pi(P)$ is externally definable. We then derive, through an inductive argument, that $\pi(P)$ is definable. By assumption (ii) it then follows that $P=\cup_{t\in \pi(P)} \{t\}\times B_t$ is definable. 

In both scenarios described above our proof of Theorem~\ref{thm:marker_steinhorn} may be significantly streamlined, including avoiding the use of Lemma~\ref{lem:preorders}, as well as the introduction and analysis of the set $I^*$. 


An issue with trying to prove the definability of $P$ in general by means of an inductive argument is that it is upfront not clear whether the projection $\pi(P)$ is externally definable in $\Mm$ \textbf{with parameters from $\mathbf{N}$}. Observe that this same issue would not be present if trying to answer Question~\ref{question:dries} positively, since the class of subsets of $M^m$ (for any $m$) definable in the pair $(\Nn, M)$ is closed under projections.

\section{Proof of the Marker-Steinhorn Theorem} \label{sec:proofMS}

\subsection{Preorders}\label{remark:preorder}

We introduce some preliminary results on definable preorders that we will need in our proof. 

Recall that a preorder is a reflexive and transitive relation. A preordered set $(B,\lleq)$ is a set $B$ together with a preorder $\lleq$ on it. It is definable if the preorder is definable. We use notation $b\prec c$ to mean $b\lleq c$ and $c \not\lleq b$. Given subsets $C,D\subseteq B$, let $C\prec D$ mean $c \prec d$ for every $c\in C$ and $d\in D$. For $b\in B$ we write $C \prec b$ and $b \prec C$ instead of $C \prec \{b\}$ and $\{b\} \prec C$ respectively. For every $b,c\in B$ let $(b,c)_{\lleq}=\{d\in B : b\prec d\prec c\}$. We also write $(b,+\infty)_{\lleq}$, $(-\infty, b]_{\lleq}$, $[b,c]_{\lleq}$, etc, with the natural meaning.

Let $p\in S_n(M)$ be an $n$-type and $\GG$ be the collection of all definable partial functions $f:M^n\rightharpoonup M\cup\{-\infty, +\infty\}$ whose domain is in $p$. Then $p$ induces a linear preorder $\lleq$ on $\GG$ given by $f \lleq g$ if and only if $\{x \in M^n : f(x)\leq g(x)\}\in p$. In other words, $f\lleq g$ when $f(\xi)\leq g(\xi)$ for some (every) realization $\xi$ of $p$.


Let $\FF=\{f_b : b\in B\}$ be a definable family of functions in $\GG$. Without loss of clarity we will often abuse notation and refer to $\lleq$ too as the linear preorder on the index set $B$ given by $b\lleq c$ if and only if $f_b \lleq f_c$. Note that, for any set $A\subseteq M$, if $\FF$ and $p$ are $A$-definable, then $\lleq$ on $B$ is $A$-definable too. 


Given a definable linearly preordered set $(B,\lleq)$ by a \emph{cut} $(P,Q)$ we mean a partition $\{P,Q\}$ of $B$ where $P \prec Q$. In practice we will consider the case where $P$ is non-empty and does not have a maximum. Hence the cut $(P,Q)$ may be identified with the partial type $p$ of intervals $(b_1,b_2)_\lleq$, for $b_1 \in P \cup \{-\infty\}$ and $b_2 \in Q \cup \{+\infty\}$. In this sense $(P,Q)$ is definable if $p$ is definable, equivalently if $P$ and/or $Q$ are definable. We will reduce the proof of the Marker-Steinhorn Theorem to a question of definability of cuts. 

Given a linearly preordered set $(B,\lleq)$ and a set $P\subseteq B$, we say that a subset $C\subseteq P$ is cofinal in $P$ if, for every $b\in P$, there exists $c\in C$ with $b \lleq c$.

The next lemma will provide an alternative in the proof of Theorem~\ref{thm:marker_steinhorn} to the use of ``$j$-gap points" in~\cite{mark_stein_94}.

\begin{lemma}\label{lem:j-cuts}
Let $(B,\lleq)$ be a definable linear preorder and $(P,Q)$ be a cut. Fix $m=\dim B$, and suppose that every definable subset $C \subseteq B$ satisfying that $C\cap P$ is cofinal in $P$ has dimension $m$. 

Let $\{C_x : x\in X\}$ be a definable family of pairwise disjoint subsets of $B$, each of dimension less than $m$. Then there exists a finite definable partition \mbox{$\{B^{(i)} : i\leq k\}$} of $B$ such that, for each $x \in X$ and $i\leq k$, either $C_x \cap B^{(i)} \subseteq P$ or $C_x \cap B^{(i)} \subseteq Q$. 
\end{lemma}
\begin{proof}
We may assume that $P\neq\emptyset$ since otherwise the result is trivial. It clearly suffices to find a finite definable partition of $\cup_{x\in X} C_x$ as described. 

For each $b\in B$ and $x\in X$, let $\Left(x,b)=\{ c \in C_x : c \prec b\}$ and $\Right(x,b)=\{ c \in C_x : b \prec c\}$. Let $I(x,b)=\{ c\in B : \Left(x,b) \prec c \prec \Right(x,b)\}$, that is, $I(x,b)$ denotes all the points in $B$ that realize the same cut over $C_x$ as $b$.

For every $x\in X$, since $\dim C_x < m$, then $C_x \cap P$ is not cofinal in $P$, and so there exists some $b \in P$ such that $\{ c \in P : b \lleq c \} \cap C_x = \emptyset$. In particular $b$ satisfies that $\Left(x,b)=P\cap C_x$ and $\Right(x,b)=Q\cap C_x$. Observe that $\{ c \in P : b \lleq c \} \subseteq I(x,b)$. Consequently $I(x,b)\cap P$ is cofinal in $P$, meaning that $\dim I(x,b)=m$. For any $x \in X$ let $B(x)=\{ b \in B : \dim I(x,b) = m\}$. We have shown that for each $x\in X$ the set $B(x)$ is non-empty. 

\begin{claim}\label{claim:j-cuts}
For each $x\in X$ the family $\{ I(x, b) : b\in B(x)\}$ is finite. 
\end{claim}
\begin{claimproof}
For each $x$ and $b, b' \in B$ observe that either $I(x,b)=I(x,b')$ or $I(x,b) \cap I(x,b') =\emptyset$. 
We fix $x\in X$. Consider the definable equivalence relation on $B$ whose equivalence classes are the sets $I(x,b)$, for $b\in B(x)$, and the complement in $B$ of their union. Recall that $\dim B=m$. Since $\dim I(x,b) =m$ for every $b\in B(x)$, we derive from~\cite[Proposition 1.8]{pillay88} that there can only be finitely many classes of this form, and the claim follows.
\end{claimproof}
 
By Claim~\ref{claim:j-cuts} and a standard compactness argument we may fix a $k>0$ such that, for each $x\in X$, the family $\{ I(x, b) : b\in B(x)\}$ has size less that $k$. Note that, for each $x\in X$ and $b\in B$, the set $I(x,b)$ completely determines the set $\Left(x,b)$ and vice versa. Hence, for any $x\in X$, the family of sets $\mathcal{P}(x)=\{ \Left(x,b) : b\in B(x)\}$ has size less than $k$. Observe moreover that $\mathcal{P}(x)$ is nested, that is, for any two sets in $\mathcal{P}(x)$ one is a subset of the other. 

For any $x \in X$ and $i\leq k$, let $\Left_x^{(i)}$ denote the $i$-th smallest set (ordered by inclusion) in $\mathcal{P}(x)$, if $\mathcal{P}(x)$ has at least $i$ sets, and otherwise let it simply denote $C_x$. Let us also define $\Left_x^{(0)} = \emptyset$. Then, for any $x\in X$ and $i\leq k$, let $C_x^{(i)}=\Left_x^{(i)} \setminus \Left_x^{(i-1)}$. Note that, for fixed $x\in X$, the sets $C_x^{(i)}$ for $i\leq k$ are pairwise disjoint and they partition $C_x$. Moreover for fixed $i$ the sets $C_x^{(i)}$ are definable uniformly in $x\in X$. For each $i\leq k$ we define 
$B^{(i)}=\cup_{x \in X} C_x^{(i)}$. Note that $\{ B^{(i)} : i\leq k\}$ is a finite definable partition of $\cup_{x\in X} B_x$. 

Towards a contradiction suppose that there exists $x\in X$ and $i\leq k$ such that $C_x \cap B^{(i)}=C_x^{(i)}$ intersects both $P$ and $Q$. Let $b\in B(x)$ be as described in the paragraph right above Claim~\ref{claim:j-cuts}, that is $\Left(x,b) = P\cap C_x$ and $\Right(x,b)= Q\cap C_x$, and $\dim I(x,b)=m$. Hence we have that $\Left(x,b) \cap C_x^{(i)} \neq \emptyset$ and $\Right(x,b) \cap C_x^{(i)} = C_x^{(i)} \setminus \Left(x,b) \neq \emptyset$.
Let $j\leq k$ be such that $\Left(x,b)=\Left_x^{(j)}$. If $j\geq i$ note that $C_x^{(i)} \subseteq \Left_x^{(j)}$ and if $j<i$ then $\Left_x^{(j)} \cap C_x^{(i)} = \emptyset$. Contradiction.
\end{proof}

The following lemma, in its form for linear orders, is contained in the literature (see for example Lemma 2.1 in~\cite{onstein09}). We include a proof to remain self-contained. 

\begin{lemma}\label{lem:preorders}
Let $(B,\lleq)$ be a definable linear preorder with $B\subseteq M$. Then there is a finite partition $\mathcal{J}$ of $B$ into points and intervals such that, for every $J\in \mathcal{J}$, exactly one of the following three conditions holds. 
\begin{enumerate}[(i)]
    \item $b\prec c$ for every $b < c$ in $J$.
    \item $b \succ c$ for every $b < c$ in $J$.
    \item $b \lleq c$ and $c\lleq b$ for every $b$ and $c$ in $J$. 
\end{enumerate}
\end{lemma}
\begin{proof}

Since the boundary points of $B$ can be taken as singletons in $\mathcal{J}$, we may assume that $B$ is a union of open intervals. For each $b\in B$, the sets $[b,+\infty)_\lleq$ and $(-\infty, b]_\lleq$ form a definable covering of $B$. Hence by o-minimality for each $b\in B$ there exists $c>b$ such that at least one of the following holds: 
\begin{enumerate}[(1)]
    \item $b \preceq (b,c)$,
    \item $b \succeq (b,c)$.
\end{enumerate}
Let $J(i)$, for $i\in \{1,2\}$, denote the definable set of points in $B$ satisfying condition $(i)$ above. So $J(1)\cup J(2) = B$. Furthermore, let $f_1:J(1)\rightarrow M\cup\{+\infty\}$ be the definable function such that $f(b)$ denotes the supremum (with respect to $<$) of all points $c>b$ such that equation $(1)$ holds. Let $f_2$ on $J(2)$ be defined analogously. By o-minimality let $\mathcal{J}$ be a finite partition of $B$ into points and intervals that is compatible with $\{ J(1), J(2)\}$ and such that, for every $J\in \mathcal{J}$ and $i\in \{1,2\}$, if $J\subseteq J(i)$ then the restriction $f_i|_J$ is continuous. We show that $\mathcal{J}$ is the desired partition.  

Let $J\in \mathcal{J}$ be an interval. We show if $J\subseteq J(1)$ then $b \lleq c$ for every $b < c$ in $J$. In particular this implies that, if $J \cap J(2) =\emptyset$, then every $b<c$ in $J$ must in fact satisfy that $b \prec c$, since otherwise we would have $b \succeq c'$ for every $c'\in (b,c]$, and thus $b\in J(2)$.   

Hence suppose that $J\subseteq J(1)$. Towards a contradiction let us fix $b<c$ in $J$ with $b \succ c$. Then by definition of $f_1$ it must be that $b < f_1(b)\leq c$. Furthermore every $d\in (b,f_1(b))$ must satisfy that $f_1(d) \leq f_1(b)$. By continuity of $f_1$ on $J$ we derive that $f_1(f_1(b))=f_1(b)$, contradicting the definition of $f_1$ and $J(1)$.

An analogous argument shows that if $J\subseteq J(2)$ then $b \succeq c$ for every $b < c$ in $J$, and in fact $b \succ c$ when $J\cap J(1)=\emptyset$, and so the lemma follows.
\end{proof}

\subsection{The proof} \label{subsec:proofMS}

We state and prove the more contentful direction of the Marker-Steinhorn Theorem, and direct the interested reader to~\cite[Corollary 2.4]{mark_stein_94} for the proof that every definable type is realized in a tame extension.

\begin{theorem}[Marker-Steinhorn Theorem~\cite{mark_stein_94}]\label{thm:marker_steinhorn}
Let $\MM$ be an o-minimal structure and let $\NN$ be a tame extension of $\MM$. For every $a\in N^n$, the type $\tp(a/M)$ is definable. 
\end{theorem}
\begin{proof}
Let us fix $a\in N^n$ and a formula $\varphi(a,u)$, with $|u|=m$. We must prove that the set $\varphi(a, M^m)$ is definable. We do this by induction on $m$ and $n$, where in the inductive step we assume that it holds for any $\al  n',m'\ar$ smaller than $\al n,m\ar$ in the lexicographic order. We may clearly assume that $a\notin M^n$ and $\varphi(a,M^m)\neq \emptyset$. 

The case $n=1$, for any $m$, follows easily from o-minimality and tameness. In particular, let $r_a$ be the supremum in $M\cup\{-\infty, +\infty\}$ of $(-\infty,a)\cap M$. If $r_a<a$ then $\tp(a/M)$ has a definable basis of the form $\{(r_a,t) : r_a<t,\, t\in M\}$, and otherwise it has a definable basis of the form $\{(t,r_a) : t<r_a, \, t\in M\}$.

The case $m=1$, for any $n$, is also straightforward as follows. By definition of tame extension it is easy to see that any interval $I'\subseteq N$ with endpoints in $N\cup\{-\infty, +\infty\}$ satisfies that $I'\cap M$ is definable (in $\Mm$), and so applying o-minimality it follows that the set $\varphi(a,M)=\varphi(a,N)\cap M$ is definable.

Hence onwards we assume that $n, m>1$. Let $a=\al d,e\ar\in N^{n-1}\times N$. Let $\{ \psi_i(N^{m+n}) : 1\leq i \leq l\}$, be a ($0$-definable) cell partition of $\varphi^{\text{opp}}(N^{m+n})$. For every $b \in N^m$, the set $\varphi^{\text{opp}}(b, N^n)=\varphi(N^n, b)$ is partitioned by sets $\psi_i(b,N^n)$ for $1\leq i \leq l$, each of which is either the empty set or a cell. In particular $\NN\models \varphi(a,b)$ if and only if $\NN\models \psi_i(b,a)$ for some $i$. So
to prove the theorem it suffices to pass to an arbitrary $1\leq i \leq l$ and show that $\psi_i(M^m,a)=\psi_i^{\text{opp}}(a, M^m)$ is definable. Hence, by passing from $\varphi(a,u)$ to $\psi_i^{\text{opp}}(a,u)$ if necessary, we may assume without loss of generality that all the non-empty sets of the form $\varphi(N^n, b)$, for $b \in N^m$, are cells (and in fact cells of the same kind, i.e. $(i_j)$-cells for some fixed sequence $(i_j) \in \{0,1\}^n$). 

Now by induction hypothesis $\tp(d/M)$ is definable.
Suppose that there exists a definable (over $M$) partial function $f:N^{n-1}\rightharpoonup N$ with $f(d)=e$. Then, for every $b \in N^m$, we have that $\NN\models \varphi(a,b)$ if and only if $\NN\models \exists t\, (\varphi(d,t,b)\wedge (f(d)=t))$, and so by induction hypothesis the result follows. 


Hence onwards we assume that there does not exist a definable (over $M$) partial function $f:N^{n-1}\rightharpoonup N$ with $f(d)=e$. In particular by assumptions on $\varphi$ we have that, for every $b\in N^m$, the set $\varphi(N^n, b)$ is either empty or a cell of the form $(f_b,g_b)$ for two definable ($\{b\}$-definable in $\Nn$) continuous functions $f_b$ and $g_b$.
Let 
\begin{align*}
B=\{& b \in M^{m} : \NN\models \exists t\, \varphi(d,t, b)\}.
\end{align*}
Since $\tp(d/M)$ is definable the set $B$ is definable. Clearly $\varphi(a, M^m)\subseteq B$. We show that the sets $\{ b \in B: f_b(d)<e\}$ and $\{ b \in B: g_b(d)>e\}$ are both definable. Then their intersection equals $\varphi(a, M^m)$. We present the proof for $\{ b \in B: f_b(d)<e\}$, since the proof for remaining set is analogous. 

Let $P=\{ b \in B : f_b(d)<e \}$ and $Q=\{ b \in B : f_b(d)>e \} = B \setminus P$. We must show that $P$ (equivalently $Q$) is definable. Since $\varphi(a,M^m)\neq \emptyset$ we have $P\neq \emptyset$. Let $\lleq$ be the linear preorder on $B$ induced by $\{ f_b(d) : b \in B\}$ described in Section~\ref{remark:preorder}, i.e. for $b,c\in B$, we have that $b\lleq c$ if and only if $f_b(d)\leq f_c(d)$.
By definability of $\tp(d/M)$ the preordered set $(B,\lleq)$ is definable. Note that $(P,Q)$ is a cut in $(B,\lleq)$.

Observe that, to prove the definability of $P$, it suffices to show that there exists a definable set $P'\subseteq P$ that is cofinal in $P$ (with respect to $\lleq$), since then $P=\{b \in B : b\lleq c \text{ for some } c\in P'\}$. So we may always pass to a definable subset $B'\subseteq B$ such that $B'\cap P$ is cofinal in $P$, and then prove definability of $P'=B'\cap P$. In particular one such set must always exist in any finite definable partition of $B$. Hence by o-minimal cell decomposition we may assume that $B$ is a cell.

\begin{claim}\label{claim:MS_1}
If there exists a definable subset $C\subseteq B$ with $\dim C < m$ such that $C\cap P$ is cofinal in $P$ then $P$ is definable. 
\end{claim}
\begin{claimproof}
By o-minimal cell decomposition we may assume that $C$ is a cell. 
If $C$ is a singleton then $P$ has a maximum and thus is definable so we may assume that $\dim C >0$. Hence we have that $0<\dim C <m$. Consider the injective projection $\pi_C:C\rightarrow N^{\dim(C)}$. We may apply the inductive case $\al n, \dim(C)\ar$ to the formula $f_{\pi^{-1}_C(b')}(d)<e$ and reach that 
\[
P \cap C = \{ \pi^{-1}_C(b') : \Nn\models f_{\pi^{-1}_C(b')}(d)<e\}
\]
is definable. Since $P\cap C$ is cofinal in $P$ this implies that $P$ is definable. 
\end{claimproof}

By Claim~\ref{claim:MS_1} we may assume that any definable subset of $B$ of dimension less than $m$ satisfies that its intersection with $P$ is not cofinal in $P$. In particular $\dim B=m$. Onwards let $\pi:N^m\rightarrow N$ denote the projection to the first coordinate.

We apply Lemma~\ref{lem:j-cuts} to the family of sets $\{t\} \times B_t$, for $t\in \pi(B)$, and derive that there exists a finite definable partition $\{B^{(i)} : i\leq k\}$ of $B$ such that, for every $i\leq k$ and $t \in \pi(B)$, either $\{t\} \times B^{(i)}_t \subseteq P$ or $\{t\} \times B^{(i)}_t \subseteq Q$.  By passing if necessary from $B$ to a subset $B^{(i)}$ such that $B^{(i)} \cap P$ is cofinal in $P$, we assume that $B$ has this property. That is, it holds that $P=\cup_{t \in \pi(P)} \{t\}\times B_t$. (By o-minimal cell decomposition and Claim~\ref{claim:MS_1} we may maintain the assumption that $B$ is an open cell.) Note that $\{\pi(P), \pi(Q)\}$ forms a partition of $\pi(B)$. To prove that $P$ is definable it suffices to show that $\pi(P)$ (equivalently $\pi(Q)$) is definable. 

Now let $\sqsubseteq$ denote the linear preorder on $\pi(B)$ given by $t \sqsubseteq s$ if and only if, for every $x \in B_t$, there is some $y\in B_s$ such that $\al t, x \ar \lleq \al s, y\ar$. Since $\lleq$ is definable then $\sqsubseteq$ is clearly definable too. Observe that, for any $t \in \pi(P)$ and $s\in \pi(Q)$, it must hold that $t \sqsubset s$. In other words, $(\pi(P), \pi(Q))$ is a cut in $(\pi(B),\sqsubseteq)$. Since $P\neq \emptyset$ we have $\pi(P)\neq \emptyset$. Onwards we say that a subset of $\pi(P)$ is cofinal in $\pi(P)$ if it is cofinal with respect to $\sqsubseteq$. Like we previously argued with $P$, in order to prove that $\pi(P)$ is definable it suffices to find a definable subset of $\pi(P)$ that is cofinal in $\pi(P)$.

We now introduce a number of sets definable in $\Nn$. Let $B^{\Nn}$ be the set defined in $\Nn$ by the same formula that defines $B$ in $\Mm$. More generally, we use the superscript $\Nn$ to denote an $M$-definable set (including an interval) which is being understood in $\Nn$. Let $B^*= \{ b \in B^{\Nn} : \Nn \models \exists t \varphi(d,t,b)\}$. Finally, let $P^*=\{ b \in B^* : f_b(d)<e\}$. Note that $B \subseteq B^* \subseteq B^{\Nn}$ and $P\subseteq P^*$ with $P^*\cap M^m = P$. 

We consider a partition of $\piBN$ into three $N$-definable subsets as follows. (Note that $\piBN = \pi(B^{\Nn})$.)

\begin{enumerate}[(i)]
    \item $I^*(0) = \{t \in \piBN : P^*_t = \emptyset\}$.
    \item $I^*(1) = \{t \in \piBN : P^*_t = \BNx \}$.
    \item $I^*(2) = \{t \in \piBN : \emptyset \neq P^*_t \subsetneq \BNx\}$.
\end{enumerate}

We fix $I^*=I^*(i)$, for some $0\leq i \leq 2$, satisfying that $I^*\cap \pi(P)$ is cofinal in $\pi(P)$. If $I^*\cap \pi(P)$ is definable (in $\Mm$) then $\pi(P)$, and thus $P$, are definable. 

Observe that, since $P \subseteq P^*$, it must hold that $I^*(0) \cap \pi(P) = \emptyset$, and so it cannot be that $I^* = I^*(0)$. Suppose that $I^*=I^*(1)$. By definition of $I^*(1)$ in this case observe that $I^* \cap \pi(P) = I^* \cap M$. Applying the case $m=1$ of the proof, we derive that $I^*\cap \pi(P)$ is definable. 

Hence we are left with the case $I^*=I^*(2)$. Onwards we assume that $P$ (equivalently $\pi(P)$) is not definable and reach a contradiction, by showing that there exists a definable subset $C$ of $B$ of dimension less than $m$ such that $C\cap P$ is cofinal in $P$, and then applying Claim~\ref{claim:MS_1}.

We require the following preliminary claim, showing that we may assume that $I^*$ is $M$-definable. Onwards let $I=I^* \cap M$. Recall that, by the case $m=1$ of the proof, $I$ is definable. 

\begin{claim}\label{claim:I^*}
There exists $t_1, t_2 \in I$, with $(t_1,t_2)^{\Nn} \subseteq I^*$, such that $(t_1,t_2) \cap \pi(P)$ is cofinal in $\pi(P)$. 
\end{claim}
\begin{claimproof}
By o-minimality $I^*$ is a finite union of points and intervals (in $N$). Hence, by making $I^*$ smaller if necessary, we may assume that it is an open interval (if a singleton is cofinal in $\pi(P)$ then we immediately have that $\pi(P)$ is definable). Similarly, by applying Lemma~\ref{lem:preorders}, we may assume that $\sqsubseteq$ restricted to $I$ is either $\leq$, the reverse order of $\leq$, or the trivial preorder where any two points are equivalent. 

Now if $I^*\cap \pi(Q) = \emptyset$ then we have that $I^* \cap \pi(P)= I^*\cap \pi(B)= I^* \cap M = I$, which is definable. Hence it must hold that $I^*\cap \pi(Q) \neq \emptyset$. Since we also have $I^* \cap \pi(P) \neq \emptyset$ we derive that $\sqsubseteq$ restricted to $I$ cannot be the trivial preorder.

Suppose that $\sqsubseteq$ on $I$ coincides with $\leq$. Pick $t_1 \in \pi(P)$ and $t_2 \in \pi(Q)$. By definition of $\sqsubseteq$ we have $t_1 \sqsubset t_2$ and so $t_1 < t_2$. Furthermore since $\pi(P)$ is not definable $t_1$ cannot be a maximum in $\pi(P)$, and so $(t_1, t_2) \cap \pi(P)$ must be cofinal in $\pi(P)$. Since $I^*$ is an interval we clearly also have that $(t_1,t_2)^{\Nn} \subseteq I^*$.
The case where $\sqsubseteq$ on $I$ equals the reverse of the order $\leq$ is analogous by picking $t_1 \in \pi(Q)$ and $t_2 \in \pi(P)$. 
\end{claimproof}

By Claim~\ref{claim:I^*} after passing to a subset if necessary we assume onwards that $I^*$ is an $M$-definable open interval in $N$. (In particular $I^*=I^{\Nn}$.)  

Since $B$ is an open cell then, for every $t\in \piBN$, the fiber $\BNx$ is an open cell in $N^{m-1}$ (in particular definably connected). Consequently, for every $t\in I^*$, by definition of $I^*(2)$ it must be that 
\[
\text{bd}(P^*_t) \cap \BNx \neq \emptyset, 
\]
where $\text{bd}(P^*_t)$ denotes the topological boundary of $P^*_t$ in $N^{m-1}$. 

Consider the set, definable in $\Nn$, given by 
\[
H^*=\bigcup_{t \in I^*} \{t\} \times \text{bd}(P^*_t).
\]
By o-minimality for each $t \in I^*$ it holds that $\dim \text{bd}(P^*_t) < m-1$, and so $H^*$ has dimension less than $m$ (see Chapter 4, Proposition 1.5 and Corollary 1.6, in \cite{dries98}). Moreover, by the above paragraph, for every $t\in I^*$ it holds that $H_t^* \cap \BNx \neq \emptyset$. Hence we have reached that there exists a definable set $H^* \subseteq N^m$ in $\Nn$ of dimension less than $m$ satisfying that $H_t^*\cap \BNx \neq \emptyset$ for every $t\in I^*$. Since $\Mm \lleq \Nn$ and $I^* = I^{\Nn}$ is $M$-definable, then there must also exists a definable (in $\Mm$) set $C \subseteq M^m$ of dimension less that $m$ satisfying that $C_t \cap B_t \neq \emptyset$ for all $t \in I$. By considering $C\cap B$ in place of $C$ we may assume that $C$ is a subset of $B$. We must now show that $C\cap P$ is cofinal in $P$. We use the following claim. 

\begin{claim}\label{claim:slice-works}
For each $t\in \pi(P)$ there exists $s \in I \cap \pi(P)$ such that $\{t\} \times B_t \prec \{s\} \times B_s$.  
\end{claim}
\begin{claimproof}
Fix $t\in \pi(P)$ and suppose that an $s$ as described does not exist. Recall that, by assumptions on $B$, every $s\in \pi(B)$ satisfies that either $B_s \subseteq P_s$ or $B_s \subseteq Q_s$. In particular every $s\in \pi(Q)$ must satisfy that $\{t\} \times B_t \prec \{s\} \times B_{s}$. Hence we derive that $I \cap \pi(Q) = \{ s \in I : \{t\} \times B_t \prec \{s\} \times B_{s}\}$, meaning that $I \cap \pi(Q)$ is definable, or equivalently $I\cap \pi(P)$ is definable, and consequently $\pi(P)$ is definable. Contradiction. 
\end{claimproof}

Now fix $b \in P$. By Claim~\ref{claim:slice-works} let $s\in I \cap \pi(P)$ be such that $b \prec \{s\} \times B_s$. Recall that $B_s=P_s$. Since, by definition of $C$, it holds that $C_s \cap B_s \neq \emptyset$, we derive that there is $c \in C \cap P$ with $b \lleq c$. It follows that $C\cap P$ is cofinal in $P$. Since $\dim C < m$ we reach a contradiction by Claim~\ref{claim:MS_1}.
\end{proof} 

Van den Dries~\cite[Proposition 8.1]{dries03} proved a strengthening of the Marker-Steinhorn Theorem for o-minimal expansions of ordered fields. We end the paper by asking whether or not this result holds without the field assumption.  

\begin{question}\label{question:dries}
Let $\Mm=(M,\ldots)$ be an o-minimal structure and $\Nn$ be a tame elementary extension. Consider the pair $(\Nn,M)$, i.e. the expansion of $\Nn$ by a unary predicate for $M$.
Is every subset of $M^n$ definable in $(\Nn,M)$ also definable in $\Mm$? 
\end{question}

\paragraph{Acknowledgements}
The author thanks Anand Pillay, Dave Marker and Charles Steinhorn, for crucial conversations on their proofs of the Marker-Steinhorn Theorem, which enabled the confirmation of the existing gap.

Margaret E. M. Thomas, Matthias Aschenbrenner, and the anonymous referee provided helpful comments on the contents and presentation of this paper. In particular the latter pointed out that the case $m=1$ in the proof of Theorem~\ref{thm:marker_steinhorn} could be derived from an elementary argument. The author also thanks Kobi Peterzil and Marcus Tressl, for directing him towards some relevant literature.  

\paragraph{Funding}
During the writing of this paper the author was supported by the Graduate School and the Department of Mathematics at Purdue University, and by the UK Engineering and Physical Sciences Research Council (EPSRC) Grant EP/V003291/1.






\bibliography{mybib_types_transversals_compactness}
\bibliographystyle{alpha}
\end{document}